\documentclass[12pt,leqno,a4paper]{amsart}
\usepackage{amssymb,enumerate}
\newtheorem{thm}{Theorem}[section]
\newtheorem{lem}[thm]{Lemma}
\newtheorem{cor}[thm]{Corollary}

\newtheorem{defi}[thm]{Definition}

\begin{document}

\title{Non-solvable groups whose non-linear character degrees have the same number of different prime divisors}
\date{\today}
\author[J.Y. Guo, Y.J. Liu, Z.Y. Wu and D. Xiao]{Junying Guo, Yanjun Liu, Ziyi Wu and Di Xiao}
\address{School of Mathematics and Statistics,
   Jiangxi Normal University, Nanchang, 330022, China}
\thanks{
The first author was supported by the National Natural Science Foundation of China
(11761034) and the  Natural Science Foundation of Jiangxi Province (20232BAB202012),
and the second by the  NSFC (12171211).}



\begin{abstract} By a result of Noritzsch, a finite solvable group whose non-linear character degrees have the same set of prime divisors is meta-abelian.
In this note we investigate finite non-solvable groups whose non-linear character degrees have the same number of different prime divisors,
and show that up to an abelian direct factor, such groups are exactly
$L_2(4), L_2(8), A_7, S_7$, the central product of a cyclic $3$-group with $3.A_7$,
or the semi-direct product of $A_7$ by a cyclic $2$-group $\langle a\rangle$ such that
$a$ non-trivially acts on $A_7$ by conjugation.
As consequence, we show that only the primes $2,3,5,7$ may occur as prime divisors of their irreducible character degrees, and
that Huppert's $\rho$-$\sigma$ conjecture holds for them.
\end{abstract}

\maketitle

\pagestyle{myheadings}
\markboth{Guo, Liu, Wu and Xiao}{Non-solvable groups whose character degrees have the same number of different prime divisors}

\section{Introduction}

The study of character degrees is an interesting topic and has a long history
 in the representation theory of finite groups. It was Isaacs and Passman
 who first started the project on the structure of finite groups all of whose non-linear
complex irreducible characters have prime degrees, see \cite{IP68, IP65}.
Along the line of research, O. Manz dealt with finite groups all of whose non-linear
complex irreducible characters have prime power degrees
and classified such non-solvable groups, see \cite{Man85,Man85-}.


\vspace{1ex}

As a further step, we investigate finite groups
all of whose non-linear complex irreducible characters have the same number of different prime divisors.
Relevantly, Noritzsch studied finite solvable groups
all of whose non-linear complex irreducible characters have the same set of prime divisors
and showed that such groups are meta-abelian, see \cite[Proposition 27.9]{H} or \cite{N95}.
In this note, we focus on finite non-solvable groups.

\vspace{1ex}

Let $G$ be a finite group, ${\rm Irr}(G)$  the set of complex irreducible characters of $G$,
and ${\rm cd}(G)$ the degree set of ${\rm Irr}(G)$.
For a positive integer $n$, we denote by $\pi(n)$ the set of prime divisors of $n$
and by $|\pi(n)|$ the number of prime divisors of $n$.

 \begin{defi} \label{def:SNPD-group}
  A finite group $G$ is called a SNPD-group if for all $1\neq \chi(1)\in {\rm cd}(G)$,
 $\chi(1)$ has the same number of different prime divisors.
\end{defi}

We first classify almost simple  SNPD-groups.

\begin{thm} \label{thm:simple-main} Let $G$ be an almost simple group.
Then $G$ is a SNPD-group if and only if $G$ is isomorphic to $L_2(4)$, $L_2(8)$,  $A_7$ or $S_7$.
\end{thm}

In general, we have

\begin{thm} \label{thm:non-solvable-main} Let $G$ be a finite non-solvable group. Then
$G$ is a SNPD-group if and only if $G\cong A\times B$,
where $A$ is an abelian group and $B$ satisfies one of the following:
\begin{enumerate}
  \item[$(1)$]  $B=L_2(4), L_2(8), A_7$ or $S_7$.
  \item[$(2)$] $B$ is  the central product of a cyclic $3$-group with $3.A_7$.
  \item[$(3)$] $B$ is the semi-direct product of $A_7$ by a cyclic $2$-group $\langle a\rangle$ such that
$B/\langle a^2\rangle \cong S_7$.
\end{enumerate}
\end{thm}

Let
$$
\sigma(G)=\max \{|\pi(\chi(1))|: \chi \in \operatorname{Irr}(G)\}
$$
and
$$
\rho(G)=\{p \text { prime}: p \mid \chi(1) \text { for some } \chi \in \operatorname{Irr}(G)\} .
$$

Huppert conjectured that $|\rho(G)|$ can be bounded in terms of $\sigma(G)$ and, if $G$ is solvable, then $|\rho(G)| \leq 2 \sigma(G)$.
Nowadays it is often called Huppert's $\rho$-$\sigma$ conjecture, and is a problem of central importance in character theory.
According to Theorem \ref{thm:non-solvable-main} and its proof, we have

\begin{cor} \label{cor:number}
Let $G$ be a finite non-solvable SNPD-group. Then $\rho(G)\subseteq \{2,3,5,7\}$, and
$\sigma(G)=1$ or 2.

In particular, Huppert's $\rho$-$\sigma$ conjecture holds for non-solvable SNPD-groups.
\end{cor}

\section{Proof of Theorems}

We first prove Theorem \ref{thm:simple-main}.


\begin{proof}[Proof of Theorem \ref{thm:simple-main}] Notice that $A_5\cong L_2(4)\cong L_2(5)$.
The ``if" part is clear by Table \ref{tab:small}, where character degrees of
$L_2(4)$, $L_2(8)$, $A_7$  and $S_7$ are listed.

\begin{table}
\begin{center}
\caption{Degrees of $L_2(4)$, $L_2(8)$, $A_7$  and $S_7$} \label{tab:small}
\begin{tabular}{ll} \hline
  $S$ &  ${\rm cd}(S)$  \\ \hline
  $L_2(4)$ & $\{1,3,4,5\}$  \\
  $L_2(8)$ &  \{1,7,8,9\}  \\
  $A_7$  & $\{1,6,10,14,15,21,35\}$ \\
  $S_7$  & $\{1,6,20,14,15,21,35\}$ \\  \hline
\end{tabular}
\end{center}
\end{table}

%
%
\vspace{1ex}

For the ``only if" part, we
let $S$ be the socle of $G$ and  note that, by the Classification Theorem of Finite Simple Groups, $S$ is
an alternating group, a simple group of Lie type,
a sporadic simple group or the Tits simple group.
It is easy to check by GAP \cite{GAP} that sporadic simple groups and the Tits simple group are not a SNPD-group (see also Table \ref{degrees sporadic}).

\vspace{1ex}

Let $S$ be a simple group of Lie type. Then $S$ has the Steinberg character, whose degree is a prime power and which extends to $G$ (see \cite{Sch85,Sch92}).
So, if $G$ is a SNPD-group then all complex irreducible characters of $G$ have prime power degrees.
By \cite[Proposition B]{Wi87}, $G$ is isomorphic to $L_2(4)$ or $L_2(8)$.

\vspace{1ex}

Finally, let $S$ be the alternating group $A_n$ with $n\geq 5$.  Then ${\rm Aut}(A_n)=S_n$ if $n\neq 6$.
 The conclusion is true for $n=5,6$ or $7$ by the GAP libray \cite{GAP}.
So we may assume that $n>7$ in the following.
The complex irreducible characters of the symmetric group $S_n$ are naturally labeled by
the partitions of $n$.
Let Irr$(S_n)=\{[\lambda]\mid \lambda\vdash n\}$,
where $\lambda \vdash n$ denotes a partition of $n$. Notice that the degree of $[\lambda]$ is $\frac{n!}{\prod_{i,j}h_{i,j}}$, where
$h_{i,j}$ is the $(i,j)$-hook number, and that the restriction
$[\lambda]_{A_n}$ of $[\lambda]\in$ Irr$(S_n)$ to $A_n$ is irreducible
if and only if $\lambda$ is not self-conjugate. See \cite{James}.

\vspace{1ex}

Let $\lambda_1=(n-1,1), \lambda_2=(n-2, 1^{2})$ and $\lambda_3=(n-3, 1^3)$  so that
the degrees $d_1,d_2$ and $d_3$  of $[\lambda_1], [\lambda_2]$ and $[\lambda_3]$
are  $n-1$, $\frac{(n-1)(n-2)}{2}$ and $\frac{(n-1)(n-2)(n-3)}{6}$, respectively.
Clearly, $\lambda_i$ is not self-conjugate for all $1\leq i\leq 3$ and $n>7$,
and so $\{d_1,d_2,d_3\}\subset {\rm cd}(S)$.

\vspace{1ex}

If $n$ is even, then $\frac{n-2}{2}$ is an integer. Clearly, $n-1$ and $\frac{n-2}{2}$ are coprime to each other.
Hence $$|\pi(d_1)|=|\pi(n-1)|< |\pi(\frac{(n-1)(n-2)}{2})|=|\pi(d_2)|.$$
Therefore, if $G$ is a SNPD-group then
$n$ is odd. Furthermore,
since for $4\mid (n-1)$,  $$|\pi(d_1)|=|\pi(n-1)|= |\pi(\frac{n-1}{2})|< |\pi(\frac{n-1}{2}\cdot (n-2))|=|\pi(d_2)|,$$
it follows that $\frac{n-1}{2}$ is odd and that $n-2$ is a prime power.

\vspace{1ex}

 If $3\mid (n-1)$ then we let $\lambda_4=(n-3,3)$ so that $d_4:= [\lambda_4](1)=\frac{n(n-1)(n-5)}{6}$.
 It is clear that $n$ and $\frac{n-1}{3}$ are coprime.
Since  $\frac{n-1}{2}$ and $\frac{n-5}{2}$ are adjacent odd numbers, it follows that
$2\cdot \frac{n-1}{2}$ and $\frac{n-5}{2}$ are  coprime.  Hence
$\frac{n-1}{3}$ and $\frac{n-5}{2}$ are coprime.
In addition, since $n-(2\cdot \frac{n-5}{2})=5$, we have that
$n\cdot \frac{n-5}{2}$ can not be a prime power (since otherwise both $n$ and $\frac{n-5}{2}$ are powers of 5, which is not possible). Therefore,
$$
|\pi(n-1)|<|\pi(\frac{n-1}{3}\cdot n\cdot \frac{n-5}{2})|=|\pi(d_4)|.$$

If $3\mid (n-2)$, then $n-2$ is a power of 3, and so
$$
|\pi(n-1)|<
|\pi((n-1)\cdot \frac{n-2}{3}\cdot \frac{n-3}{2})|=|\pi(d_3)|.$$

Finally, if $3\mid (n-3)$, then
$$
|\pi(n-1)|<
|\pi((n-1)\cdot (n-2)\cdot \frac{n-3}{6})|=|\pi(d_3)|.$$
This finishes the proof.
\end{proof}

In the following we prove Theorem \ref{thm:non-solvable-main}, starting with a result of
Bianchi, Chillag, Lewis and Pacifici.

\begin{lem}\label{lem:extend} Let $N$ be a minimal normal subgroup of $G$ such that $N=S_1 \times \cdots \times S_t$,
 where $S_i \cong S$, a nonabelian simple group.
 If $\sigma \in \operatorname{Irr}(S)$ extends to ${\rm Aut}(S)$,
 then $\sigma \times  \cdots \times \sigma \in \operatorname{Irr}(N)$ extends to $G$.
\end{lem}

\begin{proof} This is \cite[Lemma 5]{BCLP07}.
 \end{proof}

The following results about the decomposition of abelian $p$-groups
might be well known. However, we add their proofs for the purpose of completeness.

\begin{lem} \label{lem:abelian-p-group} Let $P$ be an abelian $p$-group and $x\in P$ of order $p$, and $y\in P$ an element of
maximal order such that $\langle y\rangle$ contains $x$. Then
$\langle y\rangle$ is a direct factor of $P$.
\end{lem}

\begin{proof} If $P$ has only one subgroup of order $p$, then $P$ is cyclic, so the result trivially holds.
So we may assume that $P$ has at least two subgroups of order $p$. Let $Q$ be one of these subgroups different
from $\langle x\rangle$. Write $\overline{P}=P/Q$. We have $o(\overline{x})=p$, $o(\overline{y})=o(y)$, and
$\langle\overline{y} \rangle = \overline{\langle y \rangle}$.
By the inductive hypothesis, $\langle\overline{y} \rangle$ is a direct factor of $\overline{P}$,
i.e., $\overline{P}=\langle\overline{y} \rangle \times \overline{H}$ for some $\overline{H}\leq \overline{P}$.
Let $H$ be the pre-image of $\overline{H}$ under the natural epimorphism from $P$ to $\overline{P}$, so that
$(\langle y \rangle Q)\cap H=Q$.
Now $P=\langle y \rangle QH=\langle y \rangle H=\langle y \rangle\times H$, finishing the proof.
\end{proof}

\begin{lem} \label{lem:abelian-p-group-2} Let $P$ be an abelian $p$-group,
$Q$ be a maximal subgroup of $P$, and $a\in P\backslash Q$ be of maximal order
among $P\backslash Q$. Then $\langle a\rangle$ is a direct factor of $P$.
\end{lem}

\begin{proof} We first claim that $a$ is of maximal order in the whole group $P$.
Assume that $b$ has order greater than $o(a)$. Then we have $b\in Q$. However,
it follows that $ab\in P\backslash Q$ with $o(ab)>o(a)$. This contradicts
the maximality of the order of $a$ among  $P\backslash Q$, proving the claim.
Now the lemma follows by
the structural theorem of abelian $p$-groups.
 \end{proof}

\begin{lem} \label{lem:A7} Let $N$ be a normal solvable subgroup of $G$ such that $G/N\cong A_7$ or $S_7$.
If $G$ is a SNPD-group, then $G\cong A\times B$,
where $A$ is an abelian group and $B$ satisfies one of the following:
\begin{enumerate}
  \item[$(1)$]  $B=A_7$ or $S_7$.
  \item[$(2)$]  $B$ is  the central product of a cyclic $3$-group with $3.A_7$.
  \item[$(3)$] $B$ is the semi-direct product of $A_7$ by a cyclic $2$-group $\langle a\rangle$ such that
$B/\langle a^2\rangle \cong S_7$.
\end{enumerate}
\end{lem}

\begin{proof} Let $\theta\in {\rm Irr}(N)$ and $T$ the stabilizer of $\theta$ in $G$.

We first suppose that $G/N\cong A_7$. According to the Atlas \cite{Atlas}, the maximal subgroups of $A_7$
up to isomorphism are
$L_2(7), S_5, A_6$ or $\left(A_4 \times C_3\right): C_2$ of index $3\cdot 5, 3\cdot 7, 7$ or $5\cdot 7$, respectively (see also Table \ref{tb: max_A7}).

\begin{table}
\begin{center}
\caption{Maximal subgroups of $A_7$ and their indices} \label{tb: max_A7}
\begin{tabular}{|c|c|c|c|c|} \hline
  {\rm Maximal subgroup} &  $L_2(7)$ & $S_5$  & $A_6$  & $\left(A_4 \times C_3\right): C_2$    \\ \hline
   {\rm Index}     &     $3\cdot 5$ & $3\cdot 7$ &$7$  & $5\cdot 7$ \\ \hline
\end{tabular}
\end{center}
\end{table}

\vspace{1ex}

(i) Suppose that $T<G$ and that $T$ is not a maximal subgroup of $G$. Let $M$ be a maximal subgroup of $G$
containing $T$. If $M/N\cong L_2(7)$ then since $L_2(7)$ only has maximal subgroups of index $7$ or $8$,
it follows that $|G: T|$ is divisible by at least 3 different prime divisors.
The conclusion is also true for
$M/N\cong S_5$ or $A_6$ where $S_5$ (resp. $A_6$) only has maximal subgroups of index $2, 5, 2\cdot 3$ or $2\cdot 5$
(resp. $2\cdot 3, 2\cdot 5$ or $3\cdot 5$).
Now let $M/N\cong \left(A_4 \times C_3\right): C_2$. If $2\mid |M:T|$ then
$|G: T|$ is divisible by at least 3 different prime divisors. For the case where  $2\nmid |M:T|$, we have that $T/N$
is a Sylow 2-subgroup of $M/N$ and is isomorphic to $D_8$. So there is  $\widetilde{\theta}\in {\rm Irr(T\mid \theta)}$
with $2\theta(1)\mid \widetilde{\theta}(1)$, and thus there is a
 $\chi\in {\rm Irr(G\mid \theta)}$ such that $\chi(1)$
 is divisible by at least 3 different prime divisors.
 Therefore, in any case $G$
has an irreducible character of degree with at least 3 different prime divisors,
contradicting $2\cdot 3\in {\rm cd}(G)$ and that $G$ is a  SNPD-group.

\vspace{1ex}

(ii) We now suppose that $T$ is a maximal subgroup of $G$.
If $T/N \cong L_2(7)$, then using the Atlas \cite{Atlas},
either $\operatorname{cd}(T \mid \theta)=\{1,3,2 \cdot \left.3,7,2^3\right\} \cdot \theta(1)$ or $\left\{2^2, 2 \cdot 3,2^3\right\} \cdot \theta(1)$
(where $\left\{2^2, 2 \cdot 3,2^3\right\} \cdot \theta(1):=\left\{2^2\theta(1), 2 \cdot 3\theta(1),2^3\theta(1) \right\}$).
Both cases implies that $G$ is not a SNPD-group, a contradiction.

\vspace{1ex}

If $T/N\cong S_5$, we have either $\operatorname{cd}(T \mid \theta)=\left\{1,2^2, 5,2 \cdot 3\right\} \cdot \theta(1)$
or $\left\{2^2, 2 \cdot 3\right\} \cdot \theta(1)$ by the Atlas \cite{Atlas}.
Also, both cases imply that $G$ is not a SNPD-group, a contradiction.

\vspace{1ex}

If $T/N\cong A_6$, then
$\operatorname{cd}(T \mid \theta)$ is one of the four sets
$
\left\{1,5,2^3, 3^2, 2 \cdot 5\right\} \cdot \theta(1),
\left\{2^2, 2^3, 2 \cdot 5\right\} \cdot \theta(1),
\left\{3,2 \cdot 3,3^2, 3 \cdot 5\right\} \cdot \theta(1)$
or $\left\{2 \cdot 3,2^2 \cdot 3\right\} \cdot \theta(1)$ by the Atlas [3].
The former three cases obviously lead to a contradiction.
For the last case, since $|G:T|=7$ it follows that $G$
has an irreducible character of degree with at least 3 different prime divisors,
contradicting $2\cdot 3\in {\rm cd}(G)$ and that $G$ is a SNPD-group.

\vspace{1ex}

Finally, if $T/N\cong\left(A_4 \times C_3\right): C_2$, then $|G: T|=5 \cdot 7$ and $T$ contains a normal subgroup $N_1$ of $G$
such that $T/N_1 \cong S_4$. Consider a character $\widetilde{\theta} \in \operatorname{Irr}\left(N_1 \mid \theta\right)$.
Then $\operatorname{cd}(T \mid \widetilde{\theta})=\{1,2,3\} \cdot \widetilde{\theta}(1)$
or $\left\{2,2^2\right\} \cdot \widetilde{\theta}(1)$. In any case,
$G$ has an irreducible character of degree with at least 3 different prime divisors,
contradicting $2\cdot 3\in {\rm cd}(G)$ and that $G$ is a SNPD-group.

\vspace{1ex}

Now we have proved that all irreducible characters of $N$ are $G$-invariant.
According to the Atlas \cite{Atlas}, the Schur multiplier of $A_7$ is $6$ and
${\rm cd}(G)$ contains the subset
$$\{1,2\cdot 3, 2\cdot 5, 2\cdot 7, 3\cdot5, 3\cdot 7, 5\cdot 7\}\cdot \theta(1),
\{2^2, 2\cdot 7, 2^2\cdot 5, 2^2\cdot 3^2\}\cdot \theta(1),$$
$$\{2\cdot 3, 3\cdot 5, 3\cdot 7, 2^3\cdot 3\}\cdot \theta(1), {\rm or}
\ \{2\cdot 3, 2^3\cdot 3,  2^2\cdot 3^2 \}\cdot \theta(1).$$
Since $G$ is a SNPD-group, we deduce that the second possibility above can not occur,
and so $2.A_7$ and $6.A_7$ is not a quotient group of $G$.
It is clear that all irreducible characters of $N$ have degree $1$,
i.e., $N$ is abelian.
In addition, by a theorem of Brauer \cite[Theorem 6.32]{I},  we have $N\leq Z(G)$.
Now, let $E:=[G,G]$. Then $E\cong A_7$ or $3.A_7$.
If the former case occurs, then $G\cong A_7\times N$.
For the latter case, we have $G\cong EP\times H$,
where $H$ is a 3-complement of $N$, $P$ is a Sylow 3-subgroup of $N$,
and $EP$
is the central product of $E$ and $P$ with $|E\cap P|=3$.
Let $y\in P$ be an element of
maximal order such that $\langle y\rangle$ contains $E\cap P$.
By Lemma \ref{lem:abelian-p-group}, $P=\langle y\rangle\times R$
for some subgroup $R$ of $P$.
Therefore, $$G=EP\times H=E(\langle y\rangle\times R)\times H=(E\langle y\rangle)\times (R\times H),$$
and thus (2) holds (with $A=R\times H$ and $B=E\langle y\rangle$).

\vspace{1ex}

Finally, we suppose that $G/N\cong S_7$. The character degrees
of $S_7, 2.S_7, 3.S_7, 6.S_7$ can be found in the  Atlas \cite{Atlas},
see also Table \ref{tb:2.S7}.

{\small
\begin{table}
\begin{center}
\caption{Character degrees of $S_7$ and its central extensions} \label{tb:2.S7}
\begin{tabular}{cl} \hline
  {\rm Group} &  {\rm Character dergees}     \\ \hline
  $S_7$     &  $1,2\cdot 3,2\cdot 7,3\cdot 5,2^2\cdot 5,3\cdot 7,5\cdot 7$  \\
  $2.S_7$     & $1,2\cdot 3,2^3,2\cdot 7,3\cdot 5,2^2\cdot 5,3\cdot 7,2^2\cdot 7,5\cdot 7,2^2\cdot 3^2$ \\
   $3.S_7$     &  $1,2\cdot 3,2^2\cdot 3,2\cdot 7,3\cdot 5,2^2\cdot 5,3\cdot 7,2\cdot 3\cdot 5,5\cdot 7,2\cdot 3\cdot 7,2^4\cdot 3$\\
   $6.S_7$     &   $1,2\cdot 3,2^3,2^2\cdot 3,2\cdot 7,3\cdot 5,2^2\cdot 5,3\cdot 7,2^2\cdot 7,2\cdot 3\cdot 5,5\cdot 7,2^2\cdot 3^2,2\cdot 3\cdot 7,2^4\cdot 3,2^3\cdot 3^2$ \\
 \hline
\end{tabular}
\end{center}
\end{table}
}
\vspace{1ex}

Let $M/N$ be the socle of $G/N$ so that $M/N\cong A_7$ and $|G:M|=2$.
With the previous conclusion on $\operatorname{cd}(M \mid \theta)$
and Clifford's theory, it is easy to see that $\theta$
is $M$-invariant, and so $\theta(1)=1$ and all irreducible characters of $N$ are $M$-invariant.
Hence $N\leq Z(M)$.
Since $\{2\cdot 3,2^3\}\subset {\rm cd}(2.S_7)\cap {\rm cd}(6.S_7)$
and $\{2\cdot 3,2\cdot 3\cdot 5\}\subset {\rm cd}(3.S_7)$, we see that
all non-trivial central extensions of $S_7$ are not quotient groups of $G$.
This implies that $M\cong A_7\times N$.

\vspace{1ex}

Let $L\lhd G$ be such that $L\cong A_7$ and $L< L_1\lhd G$ such that $L_1\cong S_7$. If $G/L$ is not abelian, then
we have $2\in {\rm cd}(G)$. But it then follows that $G$ is not a SNPD-group, a contradiction.
Hence $G/L$ is abelian, and so $N\leq Z(G)$.
Let $N_0$ be the unique Sylow $2$-subgroup of $N$, $N_{2'}$ the unique 2-complement of $N$,
and $\widehat{N_0}$ be the subgroup of $G$ such that $G=L_1N=L\widehat{N_0}\times N_{2'}$
and $|\widehat{N_0}:N_0|=2$. In particular, $N_0$ is a maximal subgroup
of $\widehat{N_0}$.
Let $a\in \widehat{N_0}\backslash N_0$ be of maximal order among
$\widehat{N_0}\backslash N_0$. By Lemma \ref{lem:abelian-p-group-2},
$\langle a\rangle$ is a direct factor of $\widehat{N_0}$, say $\widehat{N_0}=\langle a\rangle \times N_a$
with $N_a\leq \widehat{N_0}$.
Now we have
$$G=L\widehat{N_0}\times N_{2'}= L(\langle a\rangle \times N_a)\times N_{2'}=L\langle a\rangle \times (N_a\times N_{2'}),$$
and thus (3) holds (with $A=N_a\times N_{2'}$ and $B=L\langle a\rangle$).
This finishes the proof.
\end{proof}

Finally, we prove Theorem \ref{thm:non-solvable-main}.

\begin{proof}[Proof of Theorem \ref{thm:non-solvable-main}]

The ``if" part follows from Table
\ref{tab:small} and the facts that
${\rm cd}(B)={\rm cd}(S_7)$ (resp. ${\rm cd}(3.A_7)=\{1,6,10,14,15,21,24,35\}$)
if Theorem \ref{thm:non-solvable-main} (3) (resp. (2)) occurs.
For the ``only if" part, we let  $N$ be the largest normal solvable subgroup of $G$
and $M/N$ a nonabelian chief factor of $G$.
Then $M/N\cong S_1 \times \cdots \times S_t$, where $S_i \cong S$, a nonabelian simple group.
Notice that ${\rm cd}(G/N)\subseteq {\rm cd}(G)$.

\vspace{1ex}

We first suppose that $S$ is a sporadic simple group or the Tits simple group. By the Atlas \cite{Atlas} or the GAP library \cite{GAP},
$S$ has two ${\rm Aut}(S)$-extendible irreducible characters of degrees as listed in Table \ref{degrees sporadic}.
By Lemma \ref{lem:extend}, it is easy to see that $G$ is not a SNPD-group, a contradiction.

\begin{table}

{\tiny
\begin{center}
\caption{Degrees of the sporadic simple groups and the Tits group} \label{degrees sporadic}
\begin{tabular}{p{1cm}p{1cm}lllp{1cm}llll} \hline
   Group & Chars.     & Degrees     &     &                        &Group    &  Chars.         & Degrees    &      &  \\ \hline
  $M_{11}$ & $ \chi_{5}$  & $11$    & $= $ & $11$     &  $Fi_{22}$ & $\chi_2$  &  78  &$=$&   $2\cdot 3\cdot 13$ \\
  &  $ \chi_{10}$  & $55$    &$=$ & $5\cdot 11$       &       &  $\chi_6$  &  $3,003$  &$=$&   $3\cdot 7\cdot 11\cdot 13$ \\   \\

  $M_{12}$  & $\chi_2$     &  11  &$=$&  11     & $Fi_{23}$ & $\chi_2$  &  782  &$=$&   $2\cdot 17\cdot 23$ \\
               & $\chi_{11}$  &  66  &$=$&   $2\cdot 3\cdot 11$  &   & $\chi_5$  &  $25,806$  &$=$&   $2\cdot 3\cdot 11\cdot 17\cdot 23$ \\    \\

 $M_{22}$ & $ \chi_{2}$  & $21$ &$=$ &$3\cdot7$   &  $Fi_{24}'$ & $\chi_2$  &  $8,671$  &$=$&   $13\cdot 23\cdot 29$           \\
     & $ \chi_{8}$  & $210$ &$=$ &$2\cdot 3\cdot5\cdot 7$  &  &   $\chi_4$  &  $249,458$  &$=$&   $2\cdot 11\cdot 17\cdot 23\cdot 29$\\     \\

$M_{23}$ & $ \chi_{2}$  & $22$ &$=$ &$2\cdot11$              &$McL$ & $\chi_2$  &  22  &$=$&   $2\cdot 11$ \\
           & $ \chi_{10}$  & $770$ &$=$ & $ 2\cdot5\cdot 7\cdot 11$         &     & $\chi_5$  &  770  &$=$&  $2\cdot 5\cdot 7\cdot 11$ \\   \\

$M_{24}$ & $ \chi_{2}$  & $23$ &$=$ & $23$    & $He$ & $\chi_2$  &  51  &$=$&   $3\cdot 17$     \\
              & $ \chi_{9}$  & $483$ &$=$ &$3\cdot 7\cdot 23$         & & $\chi_13$  &  $4,080$  &$=$&   $2^4\cdot 3\cdot 5\cdot 17$ \\      \\

  $J_{1}$  &  $\chi_6$  &  77  &$=$&   $7\cdot 11$       &    $Ru$ & $\chi_2$  &  378  &$=$&   $2\cdot 3^3\cdot 7$ \\
             & $\chi_9$  &  120  &$=$&   $2^3\cdot 3\cdot 5$   &            & $\chi_6$  &  $3,276$  &$=$&   $2^2\cdot 3^2\cdot 7\cdot 13$\\      \\

  $J_2$  & $\chi_2$  &  14  &$=$&   $2\cdot 7$  & $Suz$ & $\chi_2$  &  143  &$=$&   $11\cdot 13$        \\
        & $\chi_8$  &  70  &$=$&   $2\cdot 5\cdot 7$  &  & $\chi_4$  &  780  &$=$&   $2^2\cdot 3\cdot 5\cdot 13$  \\     \\

  $J_3$   & $\chi_6$  &  324  &$=$&   $2^2\cdot 3^4$   & $O'N$ & $\chi_3$  &  $13,376$  &$=$&   $2^6\cdot 11\cdot 19$ \\
         &$\chi_{10}$ &  1,140  &$=$&   $2^2\cdot 3\cdot 5\cdot 19$ &  & $\chi_{11}$  &  $52,668$  &$=$&   $2^2\cdot 3^2\cdot 7\cdot 11\cdot 19$\\     \\

  $J_4$ & $\chi_2$  &  1,333  &$=$&   $31\cdot 43$  & $HN$ & $\chi_2$  &  133  &$=$&   $7\cdot 19$  \\
         & $\chi_4$  &  $299,367$  &$=$&   $3^2\cdot 29\cdot 31\cdot 37$  &  & $\chi_6$  &  $8,778$  &$=$&   $2\cdot 3\cdot 7\cdot 11\cdot 19$ \\     \\

  $Co_1$ & $\chi_3$  &  299  &$=$&   $13\cdot 23$   & $Ly$ & $\chi_2$  &  $2,480$  &$=$&   $2^4\cdot 5\cdot 31$  \\
        & $\chi_5$  &  8,855  &$=$&   $5\cdot 7\cdot 11\cdot 23$  &  &  $\chi_5$  &  $48,174$  &$=$&   $2\cdot 3\cdot 7\cdot 31\cdot 37$\\    \\

   $Co_2$ & $\chi_2$  &  23  &$=$&   23  & $Th$ & $\chi_2$  &  248  &$=$&  $2^3\cdot 31$      \\
          & $\chi_5$  &  $1,771$  &$=$&   $7\cdot 11\cdot 23$  &   & $\chi_6$  &  $30,628$  &$=$&   $2^2\cdot 13\cdot 19\cdot 31$  \\    \\

   $Co_3$ & $\chi_2$  &  23  &$=$&   23&    $B$ & $\chi_2$  &  $4,371$  &$=$&   $3\cdot 31\cdot 47$          \\
         & $\chi_8$  &  1,771  &$=$&   $7\cdot 11\cdot 23$   && $\chi_4$  &  $1,139,374$  &$=$&   $2\cdot 17\cdot 23\cdot 31\cdot 47$   \\    \\

    $HS$  &  $\chi_2$  &  22  &$=$&   $2\cdot 11$  &  $M$ & $\chi_2$  &  $196,883$  &$=$&   $47\cdot 59\cdot 71$  \\
 &  $\chi_{10}$  &  770  &$=$&   $2\cdot 5\cdot 7\cdot 11$  &  & $\chi_3$  &  $21,296,876$  &$=$&  $2^2\cdot 31\cdot 41\cdot 59\cdot 71$   \\     \\

  ${}^2F_4(2)'$ & $ \chi_{7}$  & $300$ &$=$ & $3\cdot5^2\cdot 2^2 $            &      & &        &   &     \\
         & $ \chi_{15}$  & $675$ &$=$ &  $3^3\cdot 5^2$          &      &       &   & &       \\  \hline

\end{tabular}
\end{center}
}
\end{table}

\vspace{1ex}

We now suppose that $S$ is a simple group of Lie type.
The Steinberg character $St$ of $S$ has prime power degree and extends to ${\rm Aut}(S)$.
By Lemma \ref{lem:extend}, $St(1)^t\in {\rm cd}(G)$.
So, if $G$ is a SNPD-group then all complex irreducible characters of $G$ have prime power degrees.
By \cite[Proposition B]{Wi87}, $G$ is isomorphic to $L_2(4)$ or $L_2(8)$ up to an abelian direct factor.

\vspace{1ex}

Finally, we suppose that $S=A_n$ with $n\geq 5$.
Since $A_5\cong L_2(4)$ and $A_6\cong L_2(9)$, we may assume that $n\geq 7$.
 With the notation in the proof of Theorem \ref{thm:simple-main}, we have that the chosen characters
are all extendible to ${\rm Aut}(A_n)=S_n$ for $n>7$.
By Lemma \ref{lem:extend}, we conclude that $\{(d_1)^t,(d_2)^t,(d_3)^t, (d_4)^t\}\subset {\rm cd(G)}$.
Therefore, for $n>7$, the proof  of Theorem \ref{thm:simple-main} shows that $G$ is not a SNPD-group, a contradiction.
Let $n=7$. If $t>1$, then $M/N$ has an irreducible character of degree divisible by $2\cdot 3\cdot 5\cdot 7$.
By Clifford's theory, $G$ has an irreducible character of degree $\chi(1)$ divisible by $2\cdot 3\cdot 5\cdot 7$.
Clearly,  $2=|\pi((d_1)^t)|<|\pi(\chi(1))|$, and so $G$ is not a SNPD-group, a contradiction.
Hence $t=1$, and so $G/N\cong A_7$ or $S_7$.
Now Lemma \ref{lem:A7} applies.
\end{proof}

\vspace{2ex}

{\bf \noindent Declaration of competing interest}

\vspace{2ex}

The authors declare that they have no known competing financial interests or personal relationships that
could have appeared to influence the work reported in this paper.


\end{document}